\documentclass{article}
\usepackage{fullpage}
\usepackage{amsfonts}
\usepackage{amsmath}
\usepackage{amssymb}
\usepackage{amsthm}
\usepackage{cite}
\usepackage{tgpagella}
\usepackage[euler-digits]{eulervm}
\author{William Schlieper}
\title{The Cohomology of Quaternionic Hyperplane Complements}
\newcommand{\RR}{\mathbold{R}}
\newcommand{\CC}{\mathbold{C}}
\newcommand{\ZZ}{\mathbold{Z}}

\newcommand{\QQ}{\mathbold{Q}}
\newcommand{\HH}{\mathbold{H}}
\newcommand{\Aa}{\mathcal{A}}
\newtheorem{proposition}{Proposition}
\newtheorem{corollary}{Corollary}

\begin{document}
\maketitle
Over the complex numbers, the complement of a collection of hyperplanes is a widely-studied object; the cohomology ring, in particular, is known to have a structure depending only on the combinatorial properties of the intersection of hyperplanes.  The fundamental group, on the other hand, requires specific knowledge about the particular embedding in complex space, though the tower of nilpotent quotients can still be determined\cite{rybnikov}.

Over the quaternions, however, since hyperplane complements are simply connected, the topological properties of hyperplane complements are simultaneously more and less complicated.  In this article, we show not only that the cohomology ring of the complement is the same algebra as in the complex case, up to a multiplication of indices by $3$, but that the rational homotopy type can be determined entirely by the cohomology ring.

In Section \ref{sec:pre}, we give precise definitions for concepts used later in the article.  In Section \ref{sec:ses}, we use topology to find a short exact sequence in the cohomology groups of the hyperplane complement.  In Section \ref{sec:osa}, we show that the Orlik-Solomon algebra is isomorphic to the cohomology algebra.  In Section \ref{sec:form}, we show that the de Rham complex, which all rational-homotopy-theoretic information can be determined from, is quasi-isomorphic to the cohomology algebra. 
\section{Preliminaries}
\label{sec:pre}
Because $\HH$ is non-commutative, we must be careful when we talk about things such as subspaces of $\HH^n$; ``left'' subspaces and ``right'' subspaces are different.  To avoid confusion, we will instead deal with such things almost entirely abstractly: instead of talking about vector spaces as we would over a field, we simply talk about $\HH$-modules; the following examples will use left $\HH$-modules, but all the other results in the paper will continue to hold if the word ``left'' is replaced with the word ``right'' (and vice versa) throughout.

All $\HH$-modules in use in this paper will be finitely generated and therefore isomorphic to $\HH^n$ for some $n$; as a left module, the action of the ring $\HH$ and its group of units $\HH^\times$ will be on the left.  However, $\HH$-linear maps $f: V \rightarrow W$ are therefore equivalent to multiplication by a quaternionic matrix \textit{on the right}.  In particular, a map $a: \HH^n \rightarrow \HH$ is equivalent to an equation of the following form:
\[a(v)=\sum_{r=1}^n v_ra_r. \]

In the following, the word subspace is used almost exclusively to mean an $\HH$-submodule; that is, a central affine subspace.  The structure of the Orlik-Solomon algebra is much simpler in the central case than in the projective or general affine cases, and all other cases may be derived from the central one, so we shall focus on it.  In particular, a hyperplane refers to the kernel of a non-zero map $a: \HH^n \rightarrow \HH$, and the intersection of any number of hyperplanes will always be non-empty.  To define a hyperplane, $a$ will be unique only up to multiplication \textit{on the right} by another non-zero quaternion, and in many uses it is often more useful to think of a hyperplane in $V$ as a one-dimensional submodule of the right $\HH$-module $V^*$, especially when working over fields, but in this case, to avoid confusion, $a$ will be treated consistently like a map.

A \textit{hyperplane configuration} or hyperplane collection, in this paper, will be used to refer simply to a finite set $\Aa$ of subspaces of a given $\HH$-module $V$. In this case, the \textit{hyperplane complement} is simply the topological space
\[X=V \setminus \bigcup_{A \in \Aa}A. \]
This topological space is the main object of study in this paper.

Another important invariant is the purely combinatorial structure given by the intersections of the subcollections of $\Aa$.  In particular, given a collection $A_1,...,A_r$, the intersection can have codimension $r$, in which case the hyperplanes intersect ``normally'', or they have codimension less than $r$; when working in $V^*$, this corresponds to linear dependence, implying that the combinatorial structure can in fact be described as a matroid, and a hyperplane collection as a representation of said matroid over $\HH^n$.

In order to determine the cohomology of a hyperplane complement, we wish to define two hyperplane collections based on any given one.  Given any hyperplane $A \in \Aa$, the deletion of $A$ in $\Aa$ is simply the collection $\Aa'=\Aa \setminus \{A\}$ with one fewer hyperplane, with the corresponding hyperplane complement $X'$ being a superset of the original complement $X=X' \setminus A$.  However, since hyperplanes will have nonempty intersection, the complement $X \setminus X'$ will not simply be $A$ but will itself have hyperplanes removed, and also be a quaternionic hyperplane complement; this complement will be the restriction of $\Aa$ to $A$, $\Aa''=\{B \cap A|B \in \Aa'\}$; note that the choice of $B \subset V$ for any given $B \cap A \subset A$ is not unique, and the complement $X''=A \cap X'$ is in fact the set difference $X \setminus X'$.
\section{The Short Exact Sequence in Cohomology}
\label{sec:ses}
Let $X$ be the complement of the hyperplane arrangement $\Aa$ in the $\HH$-module $V$, and $X'$ and $X''$ be the complements of the deletion $\Aa'$ of and restriction $\Aa''$ to a particular hyperplane $A$ corresponding to $a: V \rightarrow \HH$ in $X$.  The singular cohomology $H^\bullet(X)$ can be computed by induction to $X'$, which has fewer hyperplanes removed, and $X''$, which is in one fewer quaternionic dimension.
\begin{proposition}
  For any $i \in \ZZ$, there is a short exact sequence
  \[0 \rightarrow H^i(X') \rightarrow H^i(X) \mathop{\rightarrow}^\sigma H^{i-3}(X'') \rightarrow 0. \]

  Furthermore, if $i$ is not a multiple of $3$, $i<0$, or $i>3 \dim V$, then $H^i(X)=0$.
\end{proposition}
\begin{proof}
  The long exact sequence in cohomology implies the exactness of the sequence
  \[H^{i}(X',X) \rightarrow H^i(X') \rightarrow H^i(X) \rightarrow H^{i+1}(X',X) \rightarrow H^{i+1}(X'). \]
  Thus, we first wish to show $H^{i}(X',X)\simeq H^{i-4}(X'')$.
  Since $X=X'\setminus X''$ and all hyperplane complements are manifolds, any relative cochain on $X'$ which is zero on $X$ will be zero outside a tubular neighborhood of $X''$; call this neighborhood $Y$, and let $Y^\times=Y \setminus X''$.  Then, $Y$ and $Y^\times$ are trivial fiber bundles over $X''$ with fibers $\HH,\HH^\times$ respectively, and 
  \[H^\bullet(X',X)=H^\bullet(Y,Y^\times)\simeq H^\bullet(\HH,\HH^\times) \otimes H^\bullet(X''); \]
  however, $H^\bullet(\HH,\HH^\times)$ is trivial outside of degree $4$ where it is isomorphic to $\ZZ$, so we have $H^{i}(X',X)\simeq H^{i-4}(X'')$; in particular, we now have the exactness at $H^i(X)$ of the desired short exact sequence.

  The theorem regarding multiples of $3$ follows from induction.  The base case of the induction is a vector space $\HH^n \cong \RR^{4n} \simeq point$; the cohomology is simply $\ZZ$ at degree $0$ and $0$ elsewhere.  If $i$ is not a multiple of $3$, $i<0$, or $i>3(\dim V)$, then $H^i(X')=H^{i-3}(X'')=0$ by the induction hypothesis and we have a short exact sequence
  \[0 \rightarrow H^i(X) \rightarrow 0 \]
  implying $H^i(X)=0$.  By the multiple-of-$3$ restriction, furthermore, the long exact sequence collapses into exactly the short exact sequences we desire.
\end{proof}
This proof is similar to Theorem 3.127 in \cite{orlikterao} or Lemma 4.1 in \cite{levine}; however, the use of degree-$3$ generators in the quaternionic case means that the proof immediately produces a short exact sequence, allowing for the following two corollaries:
\begin{corollary}
  $H^i(X)$ is a free abelian group.
\end{corollary}
\begin{corollary}
  Let $P_X(t)$ be the Poincar\'{e} polynomial of $X$.  Then $P_X(t)=P_{X'}(t)+t^3P_{X''}(t)$.
\end{corollary}

\section{The Orlik-Solomon Algebra}
\label{sec:osa}

However, the above short exact sequence does not encode all the information given in the cohomology ring of $X$.  Much as in the complex case, the full ring structure of a quaternionic hyperplane arrangement is determined by its combinatorial structure through the Orlik-Solomon algebra, a commutative graded differential algebra (over any ring; without loss of generality, we use $\ZZ$) whose structure is entirely determined by which intersections of hyperplanes have the ``expected'' codimension.

Given a hyperplane arrangement $\Aa$, ``the'' exterior algebra $E(\Aa)$ is simply the algebra $\Lambda(e_A)_{A \in \Aa}$, with a differential $\partial$ of degree $-1$ given by $\partial e_A=1$.  The Orlik-Solomon ideal $I$ is the ideal generated by all elements of the form $\partial(e_{A_1}...e_{A_k})$ for all collections $\{A_1,...,A_k\} \subseteq \Aa$ such that the codimension of $A_1 \cap ... \cap A_r$ is less than $r$.  The Orlik-Solomon algebra is simply defined as $O(\Aa)=E(\Aa)/I$ with $\partial$ defined as before.

However, to be isomorphic to the cohomology ring in the quaternionic case, we must actually use $E_{\times 3}(\Aa)$ and $O_{\times 3}(\Aa)$, where the generators $e_A$ have degree $3$ and $\partial$ has degree $-3$.

\begin{proposition}
  Given a hyperplane arrangement $\Aa$ over an $\HH$-module $V$ and corresponding complement $X$, there is a surjective map of degree $-3$ commutative graded differential algebras
  \[\psi: E_{\times 3}(\Aa) \rightarrow H^\bullet(X). \]
\end{proposition}
\begin{proof}
  The $\HH$-module structure $\rho: \HH \times V \rightarrow V$ restricts to an action $\rho: \HH^\times \times X \rightarrow X$ of the multiplicative group on $X$, which passes to a map in cohomology
  \[\rho^*: H^\bullet(X) \rightarrow H^\bullet(\HH^\times) \otimes H^\bullet(X); \]

  Since $H^\bullet(\HH^\times)=\Lambda(s)$ with $s$ a single generator in degree $3$, a differential on $H^\bullet(X)$ can be defined by $\rho^*(x)=1 \otimes x + s \otimes \partial x$; the associativity of the action corresponds to $\partial^2=0$ and the product structure being preserved corresponds to the Leibniz rule.

  For any $A \in \Aa$, there is a corresponding $\HH$-linear map $a: V \rightarrow \HH$ which restricts to a $\HH^\times$-map $a: X \rightarrow \HH^\times$ and we let $\psi(e_A)=a^*(s)$; the $\HH$-linearity ensures that the maps take $\partial(\psi(e_A))$ to $1$.

  To check surjectivity, we first observe the action of the short exact sequence in cohomology on particular elements of $\psi(E_{\times 3}(\Aa))$.  Fix $A \in \Aa$ and define the corresponding deletions $\Aa',X'$, restrictions $\Aa'',X''$, and map $a: V \rightarrow \HH$.  Next, select any number $B_1,...,B_r \in \Aa'$, and set $C_l=B_l \cap A \in \Aa''$ with $b_l: V \rightarrow \HH$ restricting to $c_l: A \rightarrow \HH$.  The element $\psi(e_Ae_{B_1}...e_{B_r})$ will correspond to a particular element in $H^{3(r+1)}(X)$, which can be defined as
  \[\psi(e_Ae_{B_1}...e_{B_r})=(a,b_1,...,b_r)^*(s \otimes ... \otimes s) \]
  where
  \[(a,b_1,...,b_r):X \rightarrow (\HH^\times)^{r+1} \]
  is the restriction of an $\HH$-linear map $V \rightarrow \HH^{r+1}$; the corresponding map $X' \rightarrow \HH \times (\HH^\times)^r$ restricts to a map $(c_1,...,c_r): X'' \rightarrow (\HH^\times)^r$; passing through these will have the same effect as the connecting homomorphism, so $\sigma(\psi(e_Ae_{B_1}...e_{B_r}))=\psi(e_{C_1}...e_{C_r}) \in H^{3r}(X'')$.

  We now proceed by induction.  In the base case of no hyperplanes, there is also no cohomology.

  Fix $A \in \Aa$.  Given $x \in H^{3(r+1)}(X)$, we have some corresponding $x''=\sigma(x) \in H^{3r}(X'')$.  By induction, there is some $w'' \in E_{\times 3}(\Aa'')$ such that $\psi(w'')=x''$.  For every $C \in \Aa''$ there exists some $B \in \Aa'$ such that $B \cap A = C$, so the corresponding map $\pi: E_{\times 3}(\Aa')\rightarrow E_{\times 3}(\Aa'')$ is surjective, so there's some $w'$ such that $\pi(w')=w''$ and $\sigma(e_Aw')=x''=\sigma(x)$; thus, $x-e_Aw' \in \ker(\sigma)$ which means it lies in $H^{3(r+1)}(X')$, and by the induction hypothesis there exists an appropriate $w \in E_{\times 3}(\Aa') \subset E_{\times 3}(\Aa)$ such that $x=w+e_aw'$.
\end{proof}

\begin{proposition}
  Given a hyperplane arrangement $\Aa$ over an $\HH$-module $V$ and corresponding complement $X$, there is a surjective map of degree $-3$ commutative graded differential algebras
  \[\varphi: O_{\times 3}(\Aa) \rightarrow H^\bullet(X). \]
\end{proposition}
\begin{proof}
  To show that such a $\varphi$ exists, we merely need to show that $\psi(I)=0$.
  
  Let $A_1,...,A_r \in \Aa$.  Then, $\psi(e_{A_1}...e_{A_r})$ will pass through the inclusion map
  \[H^{3r}(V \setminus (A_1 \cup ... \cup A_r)) \rightarrow H^{3r}(X) \]
  so without loss of generality let $\Aa=\{A_1,...,A_r\}$.

  Let $W=\bigcap_{A \in \Aa}A$.  Then, $W$ acts on $X$ additively, and $X$ is homotopic to $X/W$ through a deformation retract.  If $W$ has codimension $<r$, then $X/W$ has dimension $<r$ and
  \[H^{3r}(X)=H^{3r}(X/W)=0. \]

  Therefore, $\psi(e_{A_1}...e_{A_r})=0$ when $A_1 \cap ... \cap A_r$ has codimension $<r$.

  Since $\psi$ is a map of commutative differential graded algebras,
  \[\psi(\partial(e_{A_1}...e_{A_r}))=\partial(\psi(e_{A_1}...e_{A_r}))=\partial(0)=0. \qedhere \]
\end{proof}
\begin{proposition}
  The map $\varphi: O_{\times 3}(\Aa) \rightarrow H^\bullet(X)$ defined in the previous proposition is an isomorphism.
\end{proposition}
\begin{proof}
  There exists some graded abelian group $M$ such that we have a short exact sequence
  \[0 \rightarrow M \rightarrow O_{\times 3}(\Aa) \rightarrow H^\bullet(X) \rightarrow 0. \]

  Since both $O_{\times 3}(\Aa)$ and $H^\bullet(X)$ are free abelian groups, $M$ is as well, and we have the relation

  \[P_M(t)=P_{O(\Aa)}(t^3)-P_X(t). \]

  However, in chapter 3 of \cite{orlikterao} it is shown that $P_{O(\Aa)}$ satisfies the same recurrence relation as $P_X$, so $P_M(t)=0$ and the map is an isomorphism of abelian groups, thus making it an isomorphism of commutative differential graded algebras.
\end{proof}

\section{Formality}
\label{sec:form}
Unfortunately, the cohomology ring does not provide the entirety of the homotopy-theoretic information of an object.  In particular, the de Rham complex $\Omega^\bullet(X)$ of a smooth manifold is a commutative differential graded algebra that contains more information than the de Rham cohomology, which is equivalent to the singular cohomology ring over $\RR$ or $\CC$.  However, since the de Rham complex is infinite-dimensional and homotopically equivalent manifolds have different complexes, the de Rham complex can be hard to deal with; in the more general case, there is a CDGA $A_{PL}(X)$ over $\QQ$ corresponding to any topological space, whose cohomology is the singular cohomology ring over $\QQ$, which are even less directly computable than the de Rham complex\cite{hess}.

However, the maps corresponding to homotopy equivalences are CDGA homomorphisms which pass to the identity on cohomology; by treating such maps, called quasi-isomorphisms, as homotopies, it turns out that there's an equivalence between simply-connected CW complexes up to homotopy and a certain class of CDGAs up to quasi-isomorphism; in particular, from any quasi-isomorphic CDGA one can find the torsion-free parts of the homotopy groups.  This also suggests that rational-homotopic data can be found from simpler CDGAs, the simplest of all being the cohomology itself with differential zero.  Those CDGAs, and the topological spaces corresponding to them, quasi-isomorphic to their cohomologies in such a fashion are called \textit{formal}.  Here, we will show that a quaternionic hyperplane complement is formal, and thus massively simplify the determination of its homotopy theory. \footnote{In fact, \cite{feichtneryuzvinsky} proves formality for a larger class of subspace arrangements that can be shown to include quaternionic hyperplane arrangements; however, this proof provides a direct quasi-isomorphism instead of a zig-zag of maps.}

An important fact here is that many of the things which are usually performed only on cohomology can also be performed on the full de Rham complex CDGA; in particular, the following:
\begin{proposition}
\label{prop:s3d}
  If the group $\HH^\times$ acts on a manifold $X$, then the action induces a degree $-3$ differential on the CDGA $\Omega^\bullet(X)$, which reduces to the one on cohomology, and anticommutes with the standard degree $1$ differential.
\end{proposition}
\begin{proof}
  The homology of $\HH^\times$ is the exterior algebra $\Lambda(s)$ with $s$ in degree $3$; there exists a quasi-isomorphism $\sigma: \Omega^\bullet(\HH^\times) \rightarrow \Lambda(s)$ defined in degree $0$ as $\sigma(f)=f(1)$ and in degree $3$ as
  \[\sigma(\omega)=C \oint_{S^3} \omega \]
  where $S^3$ represents the unit hypersphere in the quaternions and $C$ an arbitrary normalization factor.

  The map $\mu: S^3 \times S^3 \rightarrow S^3$ induces a map $\Omega^\bullet(\HH^\times) \rightarrow \Omega^\bullet(\HH^\times \times \HH^\times)$ which is almost but not quite the comultiplication of a (commutative differential graded) Hopf algebra; however, the quasi-isomorphism
  \[\Omega^\bullet(\HH^\times) \otimes \Omega^\bullet(\HH^\times) \rightarrow \Omega^\bullet(\HH^\times \times \HH^\times)\] 
  allows a map into
  \[\Omega^\bullet(\HH^\times \times \HH^\times) \otimes_{\Omega^\bullet(\HH^\times) \otimes \Omega^\bullet(\HH^\times)} (\Lambda(s) \otimes \Lambda(s)) \simeq \Lambda(s) \otimes \Lambda(s)\]
  which commutes with the composition of $\sigma$ and the coproduct on $\Lambda(s)$ where $\Delta(s)=s \otimes 1+1 \otimes s$.

  Given a manifold $X$ with a $\HH^\times$-action, the composition of the maps

  \[\Omega^\bullet(X) \rightarrow \Omega^\bullet(H^\times \times X) \rightarrow \Omega^\bullet(H^\times \times X) \otimes_{\Omega^\bullet(H^\times)} \Lambda(s) \simeq \Lambda(s) \otimes \Omega^\bullet(X) \]
  produces a candidate for a $\Lambda(s)$-comodule structure on $\Omega^\bullet(X)$, which turns out to in fact be a comodule structure compatible with the commutative differential graded algebra structure, which is equivalent to a differential of degree $-3$; compatibility with the corresponding maps in homology follows from the maps other than the action being quasi-isomorphisms.
\end{proof}
The de Rham complex has another important property: everything will lie in the range of gradings between $0$ and the dimension of the manifold.  In fact, given a subset $X$ of $\CC^n$, the fact that
\[\Omega^\bullet(X) \simeq \Lambda(dz_1,...,dz_n) \otimes \Lambda(\overline{dz_1},...,\overline{dz_n}) \otimes C^\infty(X) \]
as vector spaces over $\CC$ means that there is a bigrading such that $\Omega^{p,q}(X)=0$ if $p>n$ or $q>n$.

\begin{proposition}
  Given a hyperplane arrangement $\Aa$ with corresponding complement $X \subset V$, there exists an algebra homomorphism
  \[\psi: O_{\times 3}(\Aa) \rightarrow \Omega(X) \]
  such that, if $O_{\times 3}(\Aa)$ is given $0$ for its degree $1$ and its usual degree $-3$ differential, and $\Omega(X)$ given the differential from Proposition \ref{prop:s3d}, $\psi$ commutes with both differentials, and that taking the cohomology with respect to the degree $1$ differential, is a quasi-isomorphism.
\end{proposition}
\begin{proof}
  First, fix an arbitrary ring inclusion $\CC \rightarrow \HH$; this will make all $\HH$-modules into $\CC$-vector spaces, and allow the use of the double grading on $\Omega^\bullet(X)$.  Next, define the $3$ form
  \[\omega(z,w)=\frac{dz \wedge dw \wedge (\overline{z} \overline{dw}-\overline{w}\overline{dz})}{|z|^2+|w|^2} \in \Omega^{2,1}(\CC^2 \setminus (0,0)), \]
  which is closed but not exact, and after identifying $\HH$ with $\CC^2$, can be chosen as an element of $\Omega^{2,1}(\HH^\times)$ such that $\partial \omega = 1$.  

  Next, again note that each hyperplane $A$ in an arrangement can be represented by a map $a: V \rightarrow \HH$ which restricts to $a: X \rightarrow \HH^\times$; much like before, let $\psi(e_A)=a^*(\omega)$.  Since $a$ is an $\HH$-linear map, it will preserve the degree $-3$ differential, and since $\HH$-linear maps are also $\CC$-linear, it will preserve the bigrading as well, so that $\partial(\psi(e_A))=1$, $\psi(e_A) \in \Omega^{2,1}(X)$, and that $\psi(e_A)$ reduces to an appropriate value in cohomology.

  Now, all that is left is to show that $\psi$ is in fact well-defined; that is, that the value will be zero at the right places.  First, given $\{A_1,...,A_r\} \subset \Aa$ such that $r > \dim V$, we have 
  \[\psi(e_{A_1}...e_{A_r}) \in \Omega^{2r,r}(X)=0. \]
  Next, since any such product with $A_1 \cap ... \cap A_r$ of codimension $<r$ factors through a complement of dimension $<r$, the product will also be zero, and since $\psi$ preserves both differentials, the differential will also be zero, thus implying that the algebra homomorphism does in fact exist.
  By the fact that $\psi(e_A)$ corresponds to a closed form and its image in the de Rham cohomology will correspond to the one in section \ref{sec:osa}, the proof of isomorphism there will imply quasi-isomorphism here.
\end{proof}
Given 5.4 in \cite{neisendorfermiller} and the quasi-isomorphism between $A_{PL}$ and the de Rham complex, we therefore have that the rational homotopy theory as well as the real homotopy theory of any quaternionic hyperplane complement is a formal consequence of its cohomology.
\bibliography{was}{}
\bibliographystyle{alpha}
\end{document}